\documentclass{amsart}
\usepackage{amsmath,amssymb,graphicx,amsthm,bm,hyperref,mathrsfs,tikz}
\usepackage{graphicx}
\usepackage{hyperref}
\hypersetup{colorlinks=true,citecolor=blue,linkcolor=cyan}
\newtheorem{thm}{Theorem}[section]
\newtheorem{cor}[thm]{Corollary}
\newtheorem{lem}[thm]{Lemma}
\newtheorem{prop}[thm]{Proposition}
\theoremstyle{definition}

\theoremstyle{remark}

\theoremstyle{question}

\numberwithin{equation}{section}

\newcommand{\Ff}{\mathbb{F}}
\newcommand{\Z}{\mathbb{Z}}

\title[On Incidences of $\varphi$ and $\sigma$ in the Function Field Setting]{On Incidences of $\varphi$ and $\sigma$ in the Function Field Setting}
\author{Patrick Meisner}
\address{School of Mathematical Science, Tel Aviv University, Ramat Aviv, Tel Aviv, 6997801, Israel }
\email{pfmeisner@gmail.com}

\begin{document}

\begin{abstract}

Erd\H{o}s first conjectured that infinitely often we have $\varphi(n) = \sigma(m)$, where $\varphi$ is the Euler totient function and $\sigma$ is the sum of divisor function. This was proven true by Ford, Luca and Pomerance in 2010. We ask the analogous question of whether infinitely often we have $\varphi(F) = \sigma(G)$ where $F$ and $G$ are polynomials over some finite field $\Ff_q$. We find that when $q\not=2$ or $3$, then this can only trivially happen when $F=G=1$. Moreover, we give a complete characterisation of the solutions in the case $q=2$ or $3$. In particular, we show that $\varphi(F) = \sigma(G)$ infinitely often when $q=2$ or $3$.

\end{abstract}

\maketitle

\section{Introduction}\label{Intro}

\subsection{Background}

Erd\H{o}s first conjectured in \cite{Erd} that there should be infinitely many solutions to the equation $\varphi(n)=\sigma(m)$ where $\varphi$ is the Euler totient function and $\sigma$ is the sum of divisor function. This question is interesting in part because it is implied by the infinitude of two set sets of primes both of which are widely believed to be infinite: twin primes and Mersenne primes. Indeed, if we have a prime $p$ such that $p+2$ is also prime then
$$\sigma(p)= p+1 =\varphi(p+2),$$
while if we have a Mersenne prime $2^n-1$, then
$$\sigma(2^n-1) = 2^n  = \varphi(2^{n+1}).$$

This conjecture was proved by Ford, Luca and Pomerance in \cite{FLP}. Moreover, they showed that for some $\alpha>0$, there are at least $\exp((\log\log x)^\alpha)$ common values $\leq x$ of $\varphi$ and $\sigma$ for large $x$. Under a uniform version of the prime $k$-tuples conjecture, Ford and Pollack \cite{FP} were able to show that the number of common values less than $x$ of $\varphi$ and $\sigma$ is $\geq \frac{x}{(\log x)^{1+o(1)}}.$

\subsection{Function Fields}

In this paper, we are interested in the analogous question about function fields. That is if $F,G\in\Ff_q[T]$ are polynomials over the finite field $\Ff_q$, then we define
\begin{align}\label{phi}
\varphi(F) = \# (\Ff_q[T]/(F))^* = \prod_{P|F} |P|^{v_P(F)-1}(|P|-1)
\end{align}
\begin{align}\label{sigma}
\sigma(G) = \sum_{D|G} |D|
\end{align}
where for any polynomial $A\in\Ff_q[T]$, $|A| = q^{\deg(A)}$. Further, unless otherwise stated, when we consider ranging over divisors of a polynomial, we always consider only \textit{monic} divisors. Therefore, the $P$ and $D$ appearing in the definition of $\varphi$ and $\sigma$ are monic.

One thing of note is that in the function field setting, the twin prime conjecture was proved by Bender and Pollack \cite{BP} in the large $q$ limit (in fact, they just need $q$ to grow sufficiently faster than $n$). Following this, Bary-Soroker \cite{BS} proved the full Hardy-Littlewood prime $k$-tuple conjecture, in the large $q$ limit. However, even with this big hammer it doesn't seem to help us prove the infinitude of solutions to $\varphi(F)=\sigma(G)$. Indeed, if we had a prime polynomial $P$ such that $P+2$ was also prime, then
\begin{align*}
\sigma(P) & = |P|+1 = q^{\deg(P)}+1 \\
& \not= q^{\deg(P)}-1 = q^{\deg(P+2)}-1 \\
& = |P+2|-1 = \varphi(P+2).
\end{align*}

The philosophy of the connection between the integers and function fields is that a true statement in one setting should have analogous true statement in the other. While the functions defined in \eqref{phi} and \eqref{sigma} are the standard analogues in the function field setting we find that the analogous statements are almost never true.

\begin{thm}\label{Thm1}

If $q=2$ or $3$ then there are infinitely many solutions to $\varphi(F)=\sigma(G)$ with $F,G\in\Ff_q[T]$ while if $q\not=2$ or $3$, the only solution is the trivial solution $F=G=1$.

\end{thm}

This is a sharp contrast to the integer setting. Not only do we not get infinitely many solutions, for most $q$ we do not get even one coincidental non-trivial one. A key ingredient for proving Theorem \ref{Thm1} for $q\not=2,3$ is a result of Zsigmondy \cite{Zsig} on primitive prime divisors of the sequence $\{a^n-b^n\}$ (see Section \ref{keypropsec} for more on this).

The proof to Theorem \ref{Thm1} for $q=2,3$ can be done by construction. For every tuple of positive integers $\mathbf{v} = (v_0,v_1,\dots,v_n)$, define
\begin{align}
V_q(\mathbf{v}) = \{(F,G)\in\Ff_q[T] :  G = \prod_{i=1}^n P_i^{v_i}, F=P_{n+1}, \deg(P_k) = v_0\prod_{i=1}^{k-1} (v_i+1) \}.
\end{align}

\begin{lem}\label{Mainlem}

If $(F,G)\in V_2(\mathbf{v})$ such that $v_0=1$ then $\varphi(F)=\sigma(G)$. While if $(F,G)\in V_3(\mathbf{v})$ with $v_0=2$, then $\varphi(TF) = \sigma(T(T+1)G)$.

\end{lem}

Clearly, the sets described in Lemma \ref{Mainlem} are infinite. Therefore, this lemma implies Theorem \ref{Thm1} for $q=2,3$. Moreover, the sets $V_q(\mathbf{v})$ together with some finite, exceptional sets, generate all the solutions to $\varphi(F)=\sigma(G)$.

\begin{thm}\label{MainThm2}

If $q=2$ or $3$ and $\varphi(F) = \sigma(G)$, then there exists a finite set of tuples of polynomials $E_q$ such that $F=\prod_{i=0}^n F_i $, $G=\prod_{i=0}^n G_j$ with $\gcd(F_i,F_j)=\gcd(G_i,G_j)=1$, $i\not=j$, $(F_0,G_0)\in E_q$ and $(F_i,G_i) \in V_q(\mathbf{v}_i)$ for some $\mathbf{v}_i$ such that $v_{i,0}| 6$ if $q=2$ or $v_{i,0}|2$ if $q=3$.

\end{thm}

In Section \ref{exceptsetsect} we discuss the exceptional sets and the possible values of $n$ and the $\mathbf{v}$'s. We get the following corollary.

\begin{cor}\label{MainThm2cor}

With the same notation as in Theorem \ref{MainThm2}, if $q=3$ then we must have $n\leq 2$. Moreover, all possible values of $\mathbf{v}_1,\mathbf{v}_2$ such that $v_{i,0}|2$ are possible.

If $q=2$, we must have $n\leq 3$. Moreover all possible values of $\mathbf{v}_1,\mathbf{v}_2,\mathbf{v}_3$ such that $v_{i,0}|6$, $i=1,2$ and $v_{3,0}=1$ are possible except for $(2,2)$ and $(1,1,1)$.

\end{cor}

Note that the two exceptions in the case $q=2$ come from the fact that $\Ff_2$ is a very small field and hence only has $2$ polynomials of degree $1$ and only $1$ polynomial of degree $2$

\subsection{Other Formulations}

While it is widely agreed that the definition of $\sigma$ in \eqref{sigma} is the correct analogue there are, however, other analogues we may consider. Define
$$\sigma_{nm}(F) = \sum_{\substack{D|F \\ D \mbox{ non-monic} }} |D|$$
to be the sum over not necessarily monic divisors of $F$.

\begin{thm}

We have infinitely many solutions to $\varphi(F) = \sigma_{nm}(G)$ for $F,G\in\Ff_q[T]$, for any $q$.

\end{thm}

\begin{proof}

We have
$$\sigma_{nm}(T^n) = \sum_{\alpha\in\Ff_q^*}\sum_{j=0}^{n} |\alpha T^j| = (q-1)\sum_{j=0}^n q^j = q^{n+1}-1 = \varphi(P)$$
where $P$ is any prime polynomial of degree $n+1$.

\end{proof}

However, we usually restrict to considering monic divisors as being monic in the function field setting is the analogue of being positive in the integers. Therefore, the correct analogue of $\sigma_{nm}$ in the integer setting would be summing up all the positive \textit{and} negative divisors of an integer. But then it is clear that this would always yield zero and then we get no solutions to $\sigma_{nm}(n)= \varphi(m)$ for $n,m\in\Z$. So again, the analogue seems to fail.

Since we are looking at analogues of sums of divisors, another natural choice would be to do just that: sum the divisors. Thus, we can consider the new function
$$\widetilde{\sigma}(F) = \sum_{D|F}D.$$
Now, to consider incidences to $\widetilde{\sigma}$ and $\varphi$, it is clear we must modify $\varphi$ slightly in order for this question to make sense. Thus we define
$$\widetilde{\varphi}(F) = \prod_{P|F} P^{v_P(F)-1}(P-1).$$
That is, we just remove the norm function in the definition of the usual $\varphi$.

\begin{thm}\label{minithm2}

The number of solutions to $\widetilde{\varphi}(F) = \widetilde{\sigma}(G)$ for $F,G\in\Ff_q[T]$ with $\deg(F)=\deg(G)=n$ is $\gg \frac{q^n}{n^2}$ as $q$ tends to infinity.

\end{thm}

\begin{proof}

The Hardy-Littlewood Theorem for function fields (\cite{BS,BP}) states that as $q$ tends to infinity, the number of primes $P$ of a fixed degree $n$ such that $P+2$ is also prime is $\gg \frac{q^n}{n^2}$ as $q$ tends to infinity. Now, it is easy to see that $\widetilde{\sigma}(P) = \widetilde{\varphi}(P+2)$.

\end{proof}

As we mentioned above, Ford and Pollack \cite{FP}, showed that under a uniform Hardy-Littlewood conjecture, they can show that the number of solutions to $\varphi(n)=\sigma(m)$ with $n,m\leq x$ is $\geq \frac{x}{\log(x)^{1+o(1)}}$. Now, Bary-Soroker \cite{BS} gives us a uniform Hardy-Littlewood conjecture in the large $q$ limit. So it likely possible to adapt Ford and Pollack's methods to the function field setting and prove, unconditionally, that there are $\geq \frac{q^n}{n^{1+o(1)}}$ solutions to $\widetilde{\varphi}(F)=\widetilde{\sigma}(G)$.

We note that in the special case $q=2$, we get that $\widetilde{\varphi}(P) = \widetilde{\sigma}(P)$ for all primes $P$. Therefore, we get that the number of solutions in $\Ff_2[T]$ with $\deg(F)=\deg(G)=n$ will be $\geq \frac{1}{2} 2^n$, as $F=G$, with $F$ square-free will always give a solution. It would be interesting to determine if for any other $q$ we get a positive proportion of solutions to $\widetilde{\phi}(F)=\widetilde{\sigma}(G)$ with $\deg(F)=\deg(G)=n$ as $n$ tends to infinity.

\textbf{Acknowledgements:} I would like to thank Jake Chinis for initially asking me this question and for useful conversations at the early stages. I would also like to thank Zeev Rudnick for pointing me to the work of Zsigmondy and Andr\'es Jaramillo Puentes for suggesting computation tools that helped with enumerating the exceptional sets.

The research leading to these results has received funding from the European Research Council under the European Union's Seventh Framework Programme (FP7/2007-2013) / ERC grant agreement n$^{\text{o}}$ 320755.

\section{Proof of Theorem \ref{Thm1}}

\subsection{Proof of Lemma \ref{Mainlem}}\label{mainlemproof}

Let $(F,G)\in V_2(\mathbf{v})$ for some $\mathbf{v}$ such that $v_0=1$. Then
\begin{align*}
\sigma(G) & = \prod_{i=1}^n (|P_i|^{v_i}+|P_i|^{v_i-1}+\dots+|P_i|+1) = \prod_{i=1}^n \frac{|P_i|^{v_i+1}-1}{|P_i|-1} \\
& = \prod_{i=1}^n \frac{2^{\prod_{j=1}^i (v_j+1)}-1}{2^{\prod_{j=1}^{i-1} (v_j+1)}-1}\\
& = 2^{\prod_{j=1}^n (v_j+1)}-1 = \varphi(F)
\end{align*}

Let $(F,G)\in V_3(\mathbf{v})$ for some $\mathbf{v}$ such that $v_0=2$. First, we note that since $v_0=2$ all the primes dividing $F$ and $G$ have degree greater than or equal to $2$. In particular, $\gcd(F,T)=\gcd(G,T(T+1))=1$. Therefore
\begin{align*}
\sigma(T(T+1)G) & = (3+1)^2 \prod_{i=1}^n (|P_i|^{v_i}+|P_i|^{v_i-1}+\dots+|P_i|+1) = 16 \prod_{i=1}^n \frac{|P_i|^{v_i+1}-1}{|P_i|-1} \\
& = 16 \prod_{i=1}^n \frac{3^{2\prod_{j=1}^i (v_j+1)}-1}{3^{2\prod_{j=1}^{i-1} (v_j+1)}-1}\\
& = 16 \frac{3^{2\prod_{j=1}^n (v_j+1)}-1}{3^2-1} \\
& = 2 (3^{2\prod_{j=1}^n (v_j+1)}-1) = \varphi(TF)
\end{align*}

\subsection{Preliminary Lemma}

Before we continue with the proof of Theorem \ref{Thm1} we have a preliminary lemma that reduces our search down significantly.

\begin{lem}\label{prelem}

Suppose $\varphi(F)=\sigma(G)$ then $F$ must be square free. Moreover, if $q\not=2$, then the number of prime divisors of $F$ must be even.

\end{lem}

\begin{proof}

We can rewrite $\varphi(F)$ as a sum of divisors in the following way:
$$\varphi(F) = \sum_{D|F} \mu(F/D) |D|.$$
Then we notice that $\varphi(F)\equiv \mu(F) \mod{q}$. Moreover, we note that $\sigma(G)\equiv 1 \mod{q}$. Hence, if $\varphi(F)=\sigma(G)$, we must have $\mu(F) \equiv 1 \mod{q}$. The result then follows.

\end{proof}

\subsection{Key Proposition}\label{keypropsec}

For any sequence $U_1,U_2,\dots,U_n,\dots$, we will say that $U_n$ has a primitive prime divisor if there exists a prime, $p$, such that $p|U_n$ but $p \nmid U_m$ for all $m<n$. A major tool in this paper is the following result of Zsigmondy \cite{Zsig} on primitive prime divisors of a class of sequences.

\begin{thm}\label{ZsigThm}

For any $a>b$ positive, coprime integers all the elements of the sequence
$$\{a-b,a^2-b^2,\dots,a^n-b^n,\dots\}$$
have a primitive prime divisor unless $a=2$, $b=1$ and $n=6$ or $a+b$ is a power of $2$ and $n=2$.

\end{thm}

We will use this theorem to show that unless $a=2$ or $3$, an element in the set multiplicatively generated by $\{a-1,a^2-1,\dots,\}$ will have a unique decomposition. This will be instrumental in proving the absence of solutions when $q\not=2,3$.

First, recall that a multiset is a set of not necessarily distinct objects $\{x_1,\dots,x_n\}$. The multiplicity of an object $x$ is the number of $x_i=x$, with the multiplicity being $0$ if $x$ does not appear in the multiset. We say two multisets $\{x_1,\dots,x_n\}$ and $\{y_1,\dots,y_m\}$ are equal if each object occurs with the same multiplicity.

\begin{prop}\label{keyprop}

Let $a$ be any integer greater than $1$, $(n_1,\dots,n_t)$ and $(m_1,\dots,m_s)$ any tuples of positive integers such that
$$\prod_{i=1}^t (a^{n_i}-1) = \prod_{j=1}^s (a^{m_j}-1).$$
Then if $a=2$, we must have $\{n_i : n_i\nmid 6\} = \{m_j : m_j\nmid6\}$ as multisets; if $a=3$, we must have $\{n_i :  n_i \nmid 2\} = \{m_j: m_j\nmid 2\}$ as multisets; if $a\not=2,3$, we must have $\{n_1,\dots,n_t\} = \{m_1,\dots,m_s\}$ as multisets.

\end{prop}

\begin{proof}

We will begin in the case where $a\not=2,2^m-1$. Then by Theorem \ref{ZsigThm}, we get that there always exists a prime, $p$, that divides $a^n-1$ but does not divide $a^m-1$ for all such $m<n$. Define $p_n$ as the smallest such prime. Denote
$$N_0:=N = \prod_{i=1}^t (a^{n_i}-1) = \prod_{j=1}^s (a^{m_j}-1).$$
Let $k$ be largest such that $p_k|N$. Then we must have that $a^k-1 |N$. Indeed, if there were some $\ell>k$ such that $a^\ell-1|N$, then $p_\ell |N$ contradicting the maximality of $k$. Moreover, if all the $n_i,m_j <k$, then we could not have $p_k|N$ as $p_k \nmid a^m-1$ for all $m< k$. By the same reasoning we see that
$$N_1 := \frac{N}{(a^k-1)^{v_{p_k}(N)/v_{p_k}(a^k-1)}} \in \Z$$
and $p_k \nmid N_1$. Hence, we need that
$$|\{i : n_i=k\}| = |\{j : m_j = k\}| = v_{p_k}(N)/v_{p_k}(a^k-1).$$
Repeating the same process with $N_1$ multiple times we get that for any $\ell$, we must have
$$|\{i : n_i=\ell\}| = |\{j : m_j = \ell\}| = v_{p_\ell}(N)/v_{p_\ell}(a^\ell-1)$$
and thus $\{n_1,\dots,n_t\} = \{m_1,\dots,m_s\}$ as multisets.

Now, if $a=2^m-1$, again by Theorem \ref{ZsigThm}, we can define $p_n$ in the same way as long as $n\not=1,2$ and, repeating the same process, we would find that for all $\ell\not=1,2$, we would get that
$$|\{i : n_i=\ell\}| = |\{j : m_j = \ell\}| = v_{p_\ell}(N)/v_{p_\ell}(a^\ell-1).$$
In particular, we have shown that $\{n_i :  n_i \nmid 2\} = \{m_j: m_j\nmid2\}$ as multisets which finishes the case for $a=3$.

We have reduced the question down to the case where all the $n_i,m_j$ are either $1$ or $2$. Let $c_\ell,d_\ell$ be the number of $n_i,m_j$ that equal $\ell$, respectively. Then we would need
$$(a-1)^{c_1}(a^2-1)^{c_2} = (a-1)^{d_1}(a^2-1)^{d_2}.$$
Now, since $a=2^m-1$, we get
\begin{multline*}
(a-1)^{c_1}(a^2-1)^{c_2} = 2^{c_1+(m+1)c_2}(2^{m-1}-1)^{c_1+c_2} \\
 = 2^{d_1+(m+1)d_2}(2^{m-1}-1)^{d_1+d_2} = (a-1)^{d_1}(a^2-1)^{d_2}.
\end{multline*}
Therefore, as long as $m\not=2$ (or $a\not=3$), we get that
$$c_1+(m+1)c_2 = d_1+(m+1)d_2 \quad \quad \quad c_1+c_2= d_1+d_2$$
whence $c_1=d_1$ and $c_2=d_2$ and $\{n_1,\dots,n_t\} = \{m_1,\dots,m_s\}$ as multisets.

Finally, when $a=2$, using the same method, Theorem \ref{ZsigThm} as well as the observation that the primes of $2^6-1$ come from $2^2-1$ and $2^3-1$ tells us that as long as $\ell\not=1,2,3,6$, we get that
$$|\{i : n_i=\ell\}| = |\{j : m_j = \ell\}| = v_{p_\ell}(N)/v_{p_\ell}(a^\ell-1).$$
This concludes the proof.

\end{proof}

\subsection{Proof of Theorem \ref{Thm1}}

We already proved the case where $q=2,3$ in Section \ref{mainlemproof}. Therefore, let $q\not=2,3$, and suppose that $\varphi(F)=\sigma(G)$. Then, by Lemma \ref{prelem}, $F$ must be square-free with an even number of prime divisors. Therefore,
$$\varphi(F) = \prod_{P|F} (|P|-1) = \prod_{P|F} (q^{\deg(P)}-1).$$
On the other hand if we write $G = \prod P^{v_P}$, then we would have
\begin{align*}
\sigma(G) & = \prod_{P|G} \sigma(P^{v_p}) = \prod_{P|G} (|P|^{v_p}+|P|^{v_p-1} + \dots + |P|+1) \\
& = \prod_{P|G} \frac{|P|^{v_p+1}-1}{|P|-1} = \prod_{P|G} \frac{q^{(v_p+1)\deg(P)}-1}{q^{\deg(P)}-1}.
\end{align*}
Since, $\varphi(F)=\sigma(G)$, we would then need
$$\prod_{P|F}(q^{\deg(P)}-1) \prod_{P|G} (q^{\deg(P)}-1) = \prod_{P|G} (q^{(v_p+1)\deg(P)}-1).$$
By Proposition \ref{keyprop}, we get
$$\{\deg(P) : P|F\} \cup \{\deg(P): P|G\} = \{(v_p+1)\deg(P) : P|G\}$$
as multisets. However, we see that the left hand side set has a size greater than or equal the right hand side with equality if and only if $\{\deg(P):P|F\}$ is empty. That is, if and only if $F=1$. Then we would have $\sigma(G)=\varphi(1)=1$ and hence $G=1$, as well.

\section{Characterising the Solutions}

We will now characterise all the solutions to $\varphi(F)=\sigma(G)$ when $q=2$ or $3$ thus proving Theorem \ref{MainThm2}.

Let $d_2=6$ and $d_3=2$ and define
$$\widetilde{E}_q = \{ (F_0,G_0)\in\Ff_q[T] : P|F_0 \implies \deg(P)|d_q \mbox{ and } P^v||G_0 \implies (v+1)\deg(P)|d_q\}.$$
Clearly $\widetilde{E}_q$ is finite and we will show that $E_q \subset \widetilde{E}_q$.

If $F=\prod_{i=1}^n P_i$ and $G= \prod_{i=1}^m Q_i^{v_i}$ such that $\varphi(F)=\sigma(G)$ then, as in the proof of Theorem \ref{Thm1}, we get
$$\prod_{i=1}^n \left(q^{\deg(P_i)}-1\right) \prod_{i=1}^m \left(q^{\deg(Q_i)}-1\right) = \prod_{i=1}^m \left(q^{(v_i+1)\deg(Q_i)}-1\right).$$

Applying Proposition \ref{keyprop} we need that
$$\{\deg(P_i) : \deg(P_i) \nmid d_q\} \cup \{\deg(Q_i) : \deg(Q_i) \nmid d_q\} = \{(v_i+1)\deg(Q_i) : (v_i+1)\deg(Q_i)\nmid d_q\}$$
as multisets.

We see that $\{\deg(Q_i) : \deg(Q_i) \nmid d_q\} = \{(v_i+1)\deg(Q_i) : (v_i+1)\deg(Q_i)\nmid d_q\}$ as multisets if and only if both sets are empty. Thus, if $\{\deg(P_i) : \deg(P_i) \nmid d_q\}$ is empty, then all three sets are empty and we get $(F,G)\in \widetilde{E}_q$.

Hence, without lose of generality, assume $\deg(P_n)\nmid d_q$. Then there exists a $Q_{i_1}$ such that $\deg(P_n) = (v_{i_1}+1)\deg(Q_{i_1})$. If $\deg(Q_{i_1})\nmid d_q$, then there exists a $Q_{i_2}$ such that $\deg(Q_{i_1}) = (v_{i_2}+1)\deg(Q_{i_2})$. We continue this process until we find a $Q_{i_k}$ such that $\deg(Q_{i_k})|d_q$. Relabel $Q_{i_j} = Q_{n,k-j+1}$ and $v_{i_j} = v_{n,k-j+1}$, so that we get
$$\deg(P_n) = \deg(Q_{n,1})\prod_{j=1}^k (v_{n,j}+1) \quad \quad \quad \deg(Q_{n,i}) = \deg(Q_{n,1})\prod_{j=1}^{i-1} (v_{n,j}+1).$$

That is, we find that $(P_n, \prod_{i=1}^k Q_{n,i}^{v_{n,i}})\in V_q(\mathbf{v})$ for some $\mathbf{v}$ such that $v_0=\deg(Q_{n,1})|d_q$.

Repeating this process for all the $P_j$ such that $\deg(P_j)\nmid d_q$ we get our result with $E_q$ some subset of $\widetilde{E}_q$.

\section{The Exceptional Sets}\label{exceptsetsect}

Let $F = \prod_{i=0}^n F_i$, $G = \prod_{i=0}^n G_i$ such that $(F_0,G_0)\in \widetilde{E}_q$, $\gcd(F_i,F_j)=\gcd(G_i,G_j)=1$, $(F_i,G_i)\in V_q(\mathbf{v}_i)$ for some $\mathbf{v}_i = (v_{i,0},v_{i,1},\dots,v_{i,n_i})$ with $v_{i,0}|d_q$ and $\varphi(F)=\sigma(G)$. In this section we will discuss what elements of $\widetilde{E}_q$ can appear in $E_q$ as well as the possible values for $n$ and the $\mathbf{v}_i$.

We have that
$$\sigma(G) = \prod_{P|G_0} \frac{q^{(v_P+1)\deg(P)}-1}{q^{\deg(P)}-1} \prod_{i=1}^n \frac{q^{v_{i,0} \prod_{j=1}^{n_i} (v_{i,j}+1) } - 1}{ q^{v_{i,0}}-1} $$
and
$$\varphi(F) = \prod_{P|F_0} \left(q^{\deg(P)}-1\right) \prod_{i=1}^n \left(q^{v_{i,0} \prod_{j=1}^{n_i} (v_{i,j}+1) } - 1 \right).$$

Hence, we need
\begin{align}\label{excepseteq}
\prod_{P|F_0} \left(q^{\deg(P)}-1\right)  \prod_{P|G_0}\left(q^{\deg(P)}-1\right) \prod_{i=1}^n \left(q^{v_{i,0}}-1\right) = \prod_{P|G_0} \left(q^{(v_P+1)\deg(P)}-1\right)
\end{align}

Notice that the degrees of the polynomials on the left hand side of \eqref{excepseteq} all divide $d_q$. Therefore, we must have that $(v_P+1)\deg(P)|d_q$ for all $P|G_0$ as well as otherwise we would necessarily have a prime dividing the right hand side of \eqref{excepseteq} that does not divide the left hand side, by Theorem \ref{ZsigThm}.

For ease of notation, we will denote
\begin{align}\label{w_d}
\omega_d(F) = \#\{P|F : \deg(P)=d\},
\end{align}
\begin{align}\label{w_di}
\omega_{d,i}(F) = \#\{P|F : \deg(P)=d, v_P=i\}
\end{align}
and
\begin{align}\label{pi_qd}
\pi_q(d) = \#\{P\in \Ff_q[T] : \deg(P)=d\}
\end{align}
Then we can rewrite \eqref{excepseteq} in terms of linear equations in the $\omega_d$ and $\omega_{d,i}$ of $F_0,G_0,G$ where $d|d_q$. Moreover, we have the obvious inequality $\omega_{d,i}(F)\leq \omega_d(F) \leq \pi_q(d)$.

\subsection{q=3}

We will begin with the case $q=3$ as it is simpler.

Using the fact that $d_q=2$, and our observation that $(v_P+1)\deg(P)|d_q$ for all $P|G_0$, we see that $G_0$ must be a product of linear primes with exponent $1$. In particular, we see that $\omega_{1,1}(G_0)=\omega_1(G_0)$.

Now, noting that $(3-1)=2$ and $(3^2-1)=2^3$, we can rewrite \eqref{excepseteq} as
\begin{align}\label{excepseteq3}
2^{\omega_1(F_0)+\omega_1(G) + 3(\omega_2(F_0)+\omega_2(G))} = 2^{3\omega_1(G_0)}.
\end{align}

Thus we need to find the solutions to
$$\omega_1(F_0)+\omega_1(G) + 3(\omega_2(F_0)+\omega_2(G)) = 3\omega_1(G_0)$$
under the constraints that
$$\omega_1(F_0), \omega_1(G)\leq \pi_3(1)=3 \quad \quad \omega_2(F_0), \omega_2(G) \leq \pi_3(2) = 3 \quad \quad \omega_1(G_0)\leq \omega_1(G)  $$

Manually going through all the possible solutions, we find that
\begin{align*}
(\omega_1(F_0),\omega_1(G_0),\omega_1(G),\omega_2(F_0),\omega_2(G))\in \{ & (0,0,0,0,0),(2,1,1,0,0),(1,1,2,0,0),(0,1,3,0,0), \\
& (1,2,2,1,0),(1,2,2,0,1),(3,2,3,0,0),(0,2,3,1,0), \\
& (0,2,3,0,1),(0,3,3,0,2),(0,3,3,1,1),(0,3,3,2,0),\\
& (3,3,3,0,1),(3,3,3,1,0)\}
\end{align*}

We summarize the information in the following table. $E_3$ is the set of tuples $(F_0,G_0)$ such that $F_0$ and $G_0$ are in the same row. We recall that $G_0$ is always a product of linear primes, so the $Q$'s appearing in the $G_0$ column will always be linear primes. Further, the third column shows the value of $n$ while the last column gives restrictions on the possible $\mathbf{v}$ values that can occur with $\emptyset$ indicating that $n=0$ and there would be no $V_3(\mathbf{v})$ part.

\begin{center}
\begin{tabular}{ c|c|c|c }
$F_0$ & $G_0$ & $n$ &$\mathbf{v}$\\
\hline
$1$ & $1$ & $0$ & $\emptyset$ \\
$P_1P_2, \deg(P_i)=1$ & $Q$ & $0$ & $\emptyset$ \\
$P_1P_2, \deg(P_i)=i$ & $Q_1Q_2$ & $0$ & $\emptyset$ \\
$P_1P_2, \deg(P_i)=2$ & $T(T+1)(T+2)$ & $0$ & $\emptyset$\\
$T(T+1)(T+2)P, \deg(P)=2$ & $T(T+1)(T+2)$ & $0$ & $\emptyset$ \\
$P, \deg(P)=1$ & $Q$ & $1$ & $v_0=1$ \\
$P, \deg(P)=2$ & $Q_1Q_2$ & $1$ & $v_0=1$ \\
$T(T+1)(T+2)$ & $Q_1Q_2$ & $1$ & $v_0=1$ \\
$P, \deg(P)=1$ & $Q_1Q_2$  & $1$ & $v_0=2$\\
$P, \deg(P)=2$ & $T(T+1)(T+2)$ & $1$ & $v_0=2$\\
$T(T+1)(T+2)$ & $T(T+1)(T+2)$ & $1$ & $v_0=2$\\
$1$ & $Q$ & $2$ & $v_{1,0}=v_{2,0}=1$ \\
$1$ & $Q_1Q_2$ & $2$ & $v_{1,0}=1, v_{2,0}=2$ \\
$1$ & $T(T+1)(T+2)$ & $2$ & $v_{1,0}=v_{2,0}=2$ \\

\end{tabular}
\end{center}

Observing this table proves Corollary \ref{MainThm2cor} for $q=3$.

\subsection{q=2}

%
%
%
%
%
%
%

Following the same method as for $q=3$, we can use the observation $2^2-1=3$, $2^3-1=7$, $2^6-1 = 3^2\cdot7$ to get that a solution to \eqref{excepseteq} corresponds to a solution to the system of equations
$$\omega_2(F_0) + \omega_2(G) + 2\omega_6(F_0)+2\omega_6(G) = \omega_{1,1}(G_0)+2\omega_{1,5}(G_0)+2\omega_{2,2}(G_0)+2\omega_{3,1}(G_0)$$
$$ \omega_3(F_0)+\omega_3(G)+ \omega_6(F_0)+\omega_6(G) = \omega_{1,2}(G_0)+\omega_{1,5}(G_0)+\omega_{2,2}(G_0)+\omega_{3,1}(G_0)$$
subject to the restraints that
$$\omega_2(F_0),\omega_2(G) \leq \pi_2(2)=1 \quad \quad \omega_3(F_0),\omega_3(G)\leq \pi_2(3)= 2 \quad \quad \omega_6(F_0),\omega_6(G) \leq \pi_2(6) = 9$$
$$\omega_{1,1}(G_0)+\omega_{1,2}(G_0)+\omega_{1,5}(G_0)\leq \pi_2(1) = 2 \quad \quad \omega_{2,2}(G_0)\leq \omega_2(G) \quad \quad \omega_{3,1}(G_0)\leq \omega_3(G) $$

Again, we can manually find all the solutions to the above equations. We find that after observing all possible solutions we always have $n\leq 3$. Further, if $n=0$ then there is no $\mathbf{v}$; if $n=1$ then we can find a solution for all $\mathbf{v}$ such that $v_1|6$; if $n=2$, we can find a solution for all $\mathbf{v}_1,\mathbf{v}_2$ such that $v_{1,0},v_{2,0}|6$ except for $(v_{1,0},v_{2,0})=(2,2)$; if $n=3$, then we can find a solution for all $\mathbf{v}_1,\mathbf{v}_2,\mathbf{v}_3$ such that $v_{1,0},v_{2,0}|6, v_{3,0}=1$ except for $(v_{1,0},v_{2,0},v_{3,0})=(1,1,1)$. We do not write down all the possible solutions as there are too many cases and it is not enlightening to do so. However, in the table below we give an example for each possible case. For ease of notation, we will denote $P_i,Q_j$ as primes such that $\deg(P_i)=i$ and $\deg(Q_j)=j$.

\begin{center}
\begin{tabular}{c|c|c|c}
$F_0$ & $G_0$ & $n$ & $\mathbf{v}$ \\
\hline
$1$ & $1$ & $0$ & $\emptyset$ \\
$1$ & $1$ & $1$ & $v_0=1$ \\
$P_2$ & $Q_3$ & $1$ & $v_0=2$ \\
$P_2$ & $Q_2^2$ & $1$ & $v_0=3$ \\
$P_2$ & $Q_2^2Q_3$ & $1$ & $v_0=6$ \\
$1$ & $1$ & $2$ & $v_{1,0}=1,v_{2,0}=1$ \\
$P_2$ & $Q_3$ & $2$ & $v_{1,0}=1,v_{2,0}=2$ \\
$P_2$ & $Q_2^2$ & $2$ & $v_{1,0}=1,v_{2,0}=3$ \\
$P_2$ & $Q_2^2Q_3$ & $2$ & $v_{1,0}=1,v_{2,0}=6$ \\
$P_2$ & $Q_1^2Q_3$ & $2$ & $v_{1,0}=2,v_{2,0}=3$ \\
$P_2$ & $Q_1^2Q_{3,1}Q_{3,2}$ & $2$ & $v_{1,0}=2,v_{2,0}=6$ \\
$P_2$ & $Q_1^2Q_2^2$ & $2$ & $v_{1,0}=3,v_{2,0}=3$ \\
$P_2$ & $Q_1^5Q_2^2$ & $2$ & $v_{1,0}=3,v_{2,0}=6$ \\
$P_2$ & $Q_1^5Q_2^2Q_3$ & $2$ & $v_{1,0}=6,v_{2,0}=6$ \\
\end{tabular}
\end{center}

For examples with $n=3$, $F=F_0F_1F_2F_3$, $G=G_0G_1G_2G_3$, let $(F_3,G_3)\in V_2(\mathbf{v}_3)$ with $v_{3,0}=1$ and $(F_0F_1F_2,G_0G_1G_2)$ one of the examples above with $n=2$.

Even though all our examples have $F_0$ either $1$ or a prime of degree $2$ this is not always the case. For example, in the case that $n=1$ and $v_0=6$ we can choose
$$F_0 = P_2P_6 \quad \quad \quad G_0 = Q_1^5Q_2^2Q_3.$$

\bibliography{phisig}
\bibliographystyle{amsplain}

\end{document}